\documentclass{amsart}
\usepackage{amsmath}
\usepackage{amssymb}
\usepackage{amsthm}
\usepackage{enumerate}
\usepackage[pdftex]{graphicx}
\usepackage{caption}
\theoremstyle{definition}
\newtheorem{definition}{Definition}[section]
\theoremstyle{plain}
\newtheorem{lemma}[definition]{Lemma}
\newtheorem{theorem}[definition]{Theorem}

\theoremstyle{remark}

\makeatletter
\@namedef{subjclassname@2020}{
  \textup{2020} Mathematics Subject Classification}
\makeatother

\begin{document}
	
	\title[$G$-invariant definable Tietze extension theorem]{$G$-invariant definable Tietze extension theorem}
	\author[M. Fujita]{Masato Fujita}
	\address{Department of Liberal Arts,
		Japan Coast Guard Academy,
		5-1 Wakaba-cho, Kure, Hiroshima 737-8512, Japan}
	\email{fujita.masato.p34@kyoto-u.jp}
	
	\author[T. Kawakami]{Tomohiro Kawakami}
	\address{Department of Mathematics,
		Wakayama University,
		Wakayama, 640-8510, Japan}
	\email{kawa0726@gmail.com}
	
	\begin{abstract}
		A $G$-invariant version of definable Tietze extension theorem for definably complete structures is proved when a definably  compact definable topological group $G$ acts definably and continuously on the definable set.
	\end{abstract}
	
	
	
	\maketitle
	
\section{Introduction}\label{sec:intro}
Consider a definably complete expansion of an ordered field $\mathcal F=(F,<,+,\cdot,0,1,\ldots)$.
Let $X$ be a locally closed definable subset of $F^n$.
A definable Tietze extension theorem asserts that a definable continuous function $f:A \to F$ defined on a definable closed subset $A$ of $X$ has a definable continuous extension $F:X \to F$.

A group $(G,\cdot)$ is called a \textit{definable topological group} when the underlying set $G$ is a definable set and the inverse and multiplication in $G$ are definable and continuous.  
We say that $X$ is a \textit{$G$-set} when a definable topological group $G$ acts  definably and continuously on $X$.
Can we extend a $G$-invariant definable continuous function $f$ defined on a closed $G$-invariant subset of a $G$-set $X$ to $X$?
In this brief note, we give an affirmative answer when $G$ is definably compact; that is, $G$ is closed and bounded in the ambient space $F^m$.
	
\section{$G$-invariant definable Tietze extension theorem} \label{sec:Tietze}

We first prove the following key lemma:
\begin{lemma}\label{lem:G_invariance}
	Consider a definably complete expansion of an ordered group $\mathcal F=(F,<,+,0,\ldots)$.
	Let $G$ be a definably compact definable topological group and $X$ be a definable closed $G$-set.
	Let $A$ be a $G$-invariant subset of $X$ and $\varphi:X \to F$ be a definable continuous function.
	Then the definable function $\Phi: X \to F$ given by $\Phi(x)=\inf\{\varphi(gx)\;|\; g \in G\}$ is well-defined, $G$-invariant and continuous.
	In addition, for each $x \in X$, there exists $g_x \in G$ such that $\Phi(x)=\varphi(g_xx)$.
\end{lemma}
\begin{proof}
	We first show that the map $\Phi$ is well-defined and the `in addition' part of the lemma.
	We fix $x \in X$.
	Since $G$ is definably compact and $\varphi$ is continuous, the definable set $\{\varphi(gx)\;|\; g \in G\}$ is definably compact by \cite[Proposition 1.10]{M}.
	Therefore, the infimum of the set $\{\varphi(gx)\;|\; g \in G\}$ is uniquely determined and the infimum is contained in this set.
	It implies that there exists $g_x$ such that $\Phi(x)=\varphi(g_xx)$.
	We have proven the well-definedness of $\Phi$ and the existence of $g_x$.
	
	The $G$-invariance of $\Phi$ is obvious by the definition.
	The remaining task is to prove that $\Phi$ is continuous.
	Fix an arbitrary point $x_0$ in $X$ and we show that $\Phi$ is continuous at $x_0$.
	We fix an arbitrary positive element $\varepsilon>0$.
	Let $F^m$ and $F^n$ be the ambient spaces of $G$ and $X$, respectively.
	For any $x=(x_1,\ldots, x_n),\ y=(y_1,\ldots, y_n) \in F^n$, we set $|x-y|=\max_{1 \leq i \leq n}|x_i-y_i|$.
	We also define $|g-h|$ similarly for elements $g$ and $h$ in $F^m$. 
	Take a positive element $R$ in $F$, we consider the set $X_R:=\{x \in X\;|\; |x-x_0| \leq R\}$.
	It is definably compact because $X$ is closed.
	The definable set $G \times X_R$ is also definably compact.
	Consider the definable continuous function $\zeta:G \times X_R \to F$ given by $\zeta(g,x)=\varphi(gx)$.
	It is uniformly continuous by \cite[Corollary 2.8]{Fuji_compact} because its domain of definition is definably compact.
	There exists $\delta>0$ such that, for each $(g,x), (h,y) \in G \times X_R$, the inequality $|\zeta(g,x)-\zeta(h,y)|<\varepsilon$ holds whenever $|g-h|<\delta$ and $|x-y|<\delta$.
	We may assume that $\delta<R$ by taking a smaller $\delta$ if necessary.
	We easily obtain $$|\varphi(gx_0)-\varphi(gx_1)|<\varepsilon$$ for each $g \in G$ and $x_1 \in X_R$ with $|x_1-x_0|<\delta$.
	
	We fix an arbitrary element $x_1 \in X$ such that $|x_1-x_0|<\delta$.
	We want to show that $|\Phi(x_1)-\Phi(x_0)|<\varepsilon$.
	This inequality means that $\Phi$ is continuous at $x_0$.
	Note that $x_1 \in X_R$.
	We can take $g_i \in G$ such that $\Phi(x_i)=\varphi(g_ix_i)$ for $i=1,2$.
	We have $$\Phi(x_1) =\inf\{\varphi(gx_1)\;|\; g \in G\} \leq \varphi(g_0x_1)<\varphi(g_0x_0)+\varepsilon=\Phi(x_0)+\varepsilon.$$
	We get $\Phi(x_0)<\Phi(x_1)+\varepsilon$ by symmetry.
	It means that $|\Phi(x_1)-\Phi(x_0)|<\varepsilon$.
	We have proven that $\Phi$ is continuous.
\end{proof}

The following are main theorems of this paper.

\begin{theorem}[$G$-invariant definable Tietze extension theorem]\label{thm:g_inv_tietze}
	Consider a definably complete expansion of an ordered field $\mathcal F=(F,<,+,\cdot,0,1,\ldots)$.
	Let $G$ be a definably compact definable topological group and $X$ be a definable $G$-set contained in $F^n$.
	Let $A$ be a $G$-invariant closed subset of $X$ and $\varphi:A \to F$ be a $G$-invariant definable continuous function.
	Assume that $X$ is locally closed in $F^n$.
	Then there exists a $G$-invariant definable continuous extension $\Phi:X \to F$ of $\varphi$.
\end{theorem}
\begin{proof}
	We may assume that $X$ is closed in $F^n$.
	In fact, the frontier $\partial X$ of $X$ in $F^n$ is closed because $X$ is locally closed.
	Let $d:F^n \to F$ be the definable continuous map such thar $d(x)$ is the distance of $x$ to $\partial X$.
	It is obvious that the zero set of $d$ is $\partial X$.
	The map $\iota:F^n \setminus \partial X \to F^{n+1}$ given by $\iota(x)=(x,1/d(x))$ is a definable homeomorphism onto its image, and the image of $X$ under $\iota$ is closed in $F^{n+1}$. 
	Therefore, we may assume that $X$ is closed by considering $\iota(X)$ in place of $X$.
	
	Since $A$ is closed in $X$, it is also closed in $F^n$.
	Applying the original definable Tietze extension theorem \cite[Lemma 6.6]{A} to $A$, there exists a definable continuous extension $\Psi:X \to F$ of $\varphi$ which is not necessarily $G$-invariant.
	We define $\Phi:X \to F$ by $\Phi(x)=:\inf\{\Psi(gx)\;|\; g \in G\}.$
	The $\Phi$ is continuous and $G$-invariant by Lemma \ref{lem:G_invariance}.
	It is obvious that the restriction of $\Phi$ to $A$ coincides with $\varphi$.
	We have shown that the map $\Phi$ is a desired extension.
\end{proof}

\begin{theorem}\label{thm:g_invariant_zeroset}
	Consider a definably complete expansion of an ordered group.
	Let $G$ be a definably compact definable topological group and $X$ be a definable closed $G$-set.
	Let $A$ be a $G$-invariant closed subset of $X$.
	There exists a $G$-invariant definable continuous function defined on $X$ whose zero set is $A$.
\end{theorem}
\begin{proof}
	Let $\mathcal F=(F,<,+,0,\ldots)$ be the given structure.
	Consider the definable map $d:X \to F$ given by $d(x)=\inf\{|x-a|\;|\; a \in A\}$, where $|x-a|$ is defined in the same manner as the proof of Lemma \ref{lem:G_invariance}.
	Since $A$ is closed in $X$, the zero set of $d$ is $A$.
	Let $D: X \to F$ be the definable map defined by $D(x)=\inf\{d(gx)\;|\;g \in G\}$.
	It is definable, continuous and $G$-invariant by Lemma \ref{lem:G_invariance}.
	The remaining task is to show that the zero set of $D$ is $A$.
	It is obvious that the zero set of $D$ contains $A$.
	The opposite inclusion is also easy.
	Take an arbitrary $x \in X \setminus A$.
	There exists $g_x \in G$ such that $D(x)=d(g_xx)$ by the `in addition' part of Lemma \ref{lem:G_invariance}.
	We have $D(x)=d(g_xx)>0$ because $g_xx \notin A$.
	We have proven that the zero set of $D$ is $A$.
\end{proof}

\end{document}